 \documentclass[final,1p,times]{elsarticle}
\usepackage{amssymb}
\usepackage{amsmath,amsthm}

\journal{Computers $\&$ Mathematics with Applications}
\usepackage{graphicx}
\usepackage{latexsym,amsfonts,amsbsy,amssymb}
\usepackage{amsmath,amsthm}
\usepackage{color}
\usepackage{bm}
\usepackage{enumerate}
\usepackage{mathtools}
\usepackage[colorlinks, citecolor=blue]{hyperref}
\newtheorem{theorem}{Theorem}[section]
\newtheorem{lemma}[theorem]{Lemma}

\newtheorem{remark}[theorem]{Remark}

\newtheorem{applemma}{Lemma}[section]

\numberwithin{equation}{section}

\newcommand{\triplenorm}[1]{%
	\left\vert\kern-0.9pt\left\vert\kern-0.9pt\left\vert #1
	\right\vert\kern-0.9pt\right\vert\kern-0.9pt\right\vert}  

\renewcommand{\vec}[1]{\mathbf{#1}}
\begin{document}

\begin{frontmatter}




\title{Parameter-free preconditioning for nearly-incompressible linear elasticity}

\author[A]{James H. Adler}
%\ead{james.adler@tufts.edu}

\author[A]{Xiaozhe Hu}
%\ead{xiaozhe.hu@tufts.edu}

\author[B]{Yuwen Li}
% \ead{smaclachlan@mun.ca}

\author[C]{Ludmil T. Zikatanov\corref{cor1}}
% \ead{ludmil@psu.edu}

\cortext[cor1]{Corresponding author. E-mail address: ludmil@psu.edu (L. T. Zikatanov).}

\address[A]{Department of Mathematics, Tufts
  University, Medford,
  MA 02155, USA}

\address[B]{School of Mathematical Sciences,  Zhejiang University, Hangzhou, Zhejiang 310058, China}

\address[C]{Department of Mathematics,
  The Pennsylvania State University, University Park, PA 16802,
  USA}


\begin{abstract}
It is well known that via the augmented Lagrangian method, one can
solve Stokes' system by solving the nearly incompressible linear
elasticity equation. In this paper, we show that the converse holds, and approximate
the inverse of the linear elasticity operator with a convex linear
combination of parameter-free operators.  In such a way, we construct
a uniform preconditioner for linear elasticity for all values of the
Lam\'{e} parameter $\lambda\in [0,\infty)$.  Numerical results confirm that by using inf-sup stable finite-element spaces for the solution of Stokes' equations, the proposed preconditioner is robust in $\lambda$.

\end{abstract}


\end{frontmatter}



%=========================================================================================================================

\section{Introduction}
The main focus of this work is on developing and analyzing an effective preconditioning operator for the  primal formulation of linear elasticity, particularly in the incompressible limit.  For a body force, $\tilde{\vec{f}}$, acting on an isotropic elastic material, we model the displacement of the deformable media, $\vec{u}$, via the governing equation,
\begin{equation}\label{eq:primal}
{\rm div}\big(2\mu\varepsilon(\vec{u}) + \tilde{\lambda}\mathrm{tr}(\varepsilon(\vec{u}))I\big)=\tilde{\vec{f}}.
\end{equation} 
Here, $\mu$, $\tilde{\lambda}$ are Lam\'{e} parameters, ${\rm tr}$ is the trace operation for tensors, $I$ is the identity tensor, and 
  the strain tensor, $\varepsilon(\vec{u})$, is 
 given by
\[\varepsilon(\vec{u}) = \frac{1}{2}\left ( \nabla\vec{u} + (\nabla\vec{u})^\top\right ).\]  
In terms of Poisson ratio, $\nu$, and Young's modulus, $E$, the Lam\'e constants are expressed as
\[\mu = \frac{E}{2+2\nu},\quad \tilde{\lambda}=\frac{E\nu}{(1+\nu)(1-2\nu)},\quad 0\le \nu< \frac12.\]
The linearly elastic material becomes nearly incompressible when $\nu\to \frac12^{-}$ 
and $\tilde{\lambda}\to\infty$ (cf. \cite{2007Braess}). 

In the incompressible limit, traditional finite-element and finite-difference schemes suffer from volumetric/Poisson locking. This  locking phenomenon is due to the poor representation of the divergence-free vector fields \cite{2013BoffiBrezziFortin-a} in the underlying space. It is therefore not surprising that locking-free numerical schemes of linear elasticity are related to discretization methods for Stokes' equations. 
A quick look at the paper by Bramble~\cite{2003Bramble-a} reveals that the Stokes' inf-sup condition implies the fundamental 2nd Korn's inequality in elasticity.  Such ideas have led to the development of stable and accurate numerical methods for Stokes' equation via the \emph{augmented Lagrangian} formulation, see, e.g., \cite{DouglasWang,BoffiLovadina,Glowinski}. Analysis of the corresponding iterative solution techniques for the resulting linear systems are studied in~\cite{SilvesterWathen}, and the approach has been successful in a variety of related applications~\cite{2019_FarrellNS,2021_FarrellLC,2022_FarrellMHD}.

Conversely, results from solving Stokes' equations can be used for developing schemes for nearly incompressible linear elasticity.  For instance, in~\cite{2007Braess}, Braess introduces an auxiliary variable, $p={\rm div}\,\vec{u}$, and uses the stability of a perturbed Stokes' problem to derive {\it a priori} error estimates which are uniform with respect to $\tilde{\lambda}$. Further works by Sch\"{o}berl~\cite{1999Schoberl} and Carstensen~\cite{2005Carstensen} utilize similar ideas to analyze and show robustness of multigrid solvers and reliability and efficiency of {\it a posteriori} error estimation for finite-element discretizations of linear elasticity. We note that, as shown in~\cite{2021LiZikatanov},  the results derived here, combined with the operator preconditioning framework in~\cite{2004LoghinWathen-a,2011MardalWinther-a}, can be utilized to design novel {\it a posteriori} error estimators for nearly incompressible linear elasticity.

Based on this notion, the main contribution of this work is to use stable discretizations of Stokes' equations to develop a preconditioner for linear elasticity that, unlike most others, is provably robust and performs uniformly well for all values of $\tilde{\lambda}\in[0,\infty)$. The main ingredients in the construction are: (1) the action of the inverse of a standard, parameter-free, elliptic operator;  and (2) computing an $H^1$-type orthogonal projection onto the space of (discrete) divergence-free vector fields. We note here that computing the projection requires solving another parameter-free (discrete) Stokes' problem. 
% Our approach can be viewed as preconditioning of discretizations of the primal formulation of linear elasticity that can be derived from stable discretizations for the Stokes equation.  
The underlying idea comes from a simple observation concerning linear elasticity with periodic boundary conditions, under which our preconditioner reduces to the exact inverse of the linear elasticity operator. In general, the preconditioner is not the exact inverse, but is very close to it. Such claims are validated by numerical tests showing that the preconditioned linear system corresponding to discretizations of~\eqref{eq:primal} have uniformly bounded condition numbers. 

This rest of the paper is organized as follows. Section \ref{sec:notation} sets up the bilinear forms and notation used throughout the paper.  In Section \ref{sec:infsup}, an inf-sup condition and Korn inequality is established to help build a parameter-free preconditioner.  Next, the spectral equivalence result that yields the robust preconditioner is given in Section \ref{sec:spectral}.  The case of periodic boundary conditions is also considered here.  Finally, numerical results confirming the theory is shown in Section \ref{sec:results}, with concluding remarks given in Section \ref{sec:conclusion}.

\section{Preliminaries and Notation}\label{sec:notation}
Let $\Omega\subset \mathbb{R}^d$ with $d\in\{2,3\}$ be
a bounded polyhedron with Lipschitz boundary. Let $(\cdot,\cdot)$ denote the $L^2(\Omega)$ inner product, $\|\cdot\|$ the $L^2(\Omega)$ norm, $Q=L^2(\Omega)$, and $V\subset [H^1(\Omega)]^d$ be a Hilbert space. By $\langle\cdot,\cdot\rangle$ we denote the duality pairing 
between $V$ and its
dual $V'$ or $Q$ and its dual $Q^\prime$.  The
boundary of $\Omega$ is $\Gamma=\partial\Omega=\Gamma_D\cup\Gamma_N$,
where $\Gamma_D$ is a closed set with respect to $\Gamma$ with a nonzero
$(d-1)$ dimensional measure.  Further, we denote by
$[H^1_D(\Omega)]^d\subset [H^1(\Omega)]^d$ the space of vector-valued
functions on $\Omega$ with vanishing traces on $\Gamma_D$.  Often, $V=[H^1_D(\Omega)]^d$, however,
we also consider examples with periodic boundary conditions on the
unit cube in $\mathbb{R}^d$ and thus, $V$ and $Q$ will be modified accordingly. 

The
variational problem of \eqref{eq:primal} is to find $\vec{u}\in V$ such that
\begin{equation}\label{elasticity-b-forms}
	\begin{aligned}
		&     a_{\lambda}(\vec{u},\vec{v}) = a(\vec{u},\vec{v}) + \lambda b(\vec{v} , \operatorname{div} \vec{u})
		= \langle \vec{f},\vec{v}\rangle,
	\end{aligned}
\end{equation}
 for all $\vec{v}\in V$, where $\mathbf{f}:=\tilde{\mathbf{f}}/(2\mu)$, $\lambda:=\tilde{\lambda}/(2\mu)$, and
\begin{align*}
    a(\vec{u},\vec{v})&= (\varepsilon(\vec{u}),\varepsilon(\vec{v})),\\b(\vec{v},q)&=(\operatorname{div} \vec{v},q).
\end{align*}
Note that we divide the original equation by $2\mu$ and obtain a modified parameter 
\[
\lambda=\frac{\nu}{1-2\nu},\quad 0\le \nu< \frac12.
\]
The bilinear forms in~\eqref{elasticity-b-forms} define
operators $A_\lambda:V\to V'$, and $B:V\to Q'$ by
\begin{equation}\label{e:operators}
	\begin{aligned}
		 \langle A_\lambda \vec{u}, \vec{v}\rangle&:= a_\lambda( \vec{u} , \vec{v}),\\ 
		\langle B\vec{v},q\rangle &:= b( \vec{v} ,q)=(\operatorname{div} \vec{v},q).
	\end{aligned}
\end{equation}
As $\lambda=0$ is a special case, we define $A:=A_0$ with
\begin{equation}\label{a-0}
	\langle A \vec{u} ,  \vec{v} \rangle = a( \vec{u} , \vec{v} ):=\left(\varepsilon\left(\vec{u}\right),\varepsilon\left(\vec{v}\right)\right).
\end{equation}

For any two operators, $X$ and $Y$ mapping a space $V$ to its dual $V'$, we write $X\lesssim Y$ when  
$\left\langle Xv,v\right\rangle\leq C\left\langle Yv,v\right\rangle$ holds for any $v\in V$ with a generic constant $C$ depending on $\Omega$ and independent of $\lambda$ and $\mu$. By $X\eqsim Y$, we denote $X\lesssim Y$ \emph{and} $Y\lesssim X$. Then, the goal of this paper is to show that
\begin{equation}\label{eq:m-lambda-def-0}
A_{\lambda}^{-1}\eqsim \frac{\lambda}{\lambda+1}PA^{-1} + \frac{1}{\lambda+1}A^{-1}=:M_\lambda,
\end{equation}
where $P$ is the $a_\lambda(\cdot,\cdot)$-orthogonal projection onto the space of divergence-free vector fields. Clearly $P$ can be implemented by 
solving Stokes' equations.

\section{Brezzi's inf-sup condition and Korn's inequality}\label{sec:infsup}
In order to develop a robust preconditioner, we consider some properties related to the inf-sup
condition on $V\times Q$ (see
e.g.~\cite{1974Brezzi-a,1969Ladyzhenskaya-a,1991BrezziFortin-a}):
\begin{equation}\label{e:inf-sup}
	\inf_{q\in Q}\sup_{\vec{v}\in V}\frac{(\operatorname{div} \vec{v},q)}{\|\nabla \vec{v}\|\|q\|}\ge \widetilde{\beta}>0.
\end{equation}
As is shown in~\cite{2003Bramble-a}, \eqref{e:inf-sup} is
equivalent to the following inequality due to Ne\v{c}as~\cite{1967Necas-a}:
\begin{equation}\label{e:necas}
	\|\vec{u}\|  \lesssim
	\left(\|\vec{u}\|_{H^{-1}(\Omega)}^2 + \sum_{j=1}^d
	\left\|\frac{\partial \vec{u}}{\partial x_j}\right\|_{H^{-1}(\Omega)}^2 \right)^{1/2}
\end{equation}
In addition,~\cite{2003Bramble-a} shows that~\eqref{e:necas} implies Korn's inequality:
\begin{equation}\label{e:korn-0}
	% \|u\|^2+\|\nabla u\|^2
	% \le
	% C\left(\|u\|^2 + \|\varepsilon(u)\|^2\right)
	\|\nabla \vec{u}\|
	\lesssim \|\vec{u}\| + \|\varepsilon(\vec{u})\|, \quad \forall \vec{u}\in [H^1(\Omega)]^d.
\end{equation}
The classical Korn's inequality is found in~\cite{1908KornA-aa,1909KornA-aa}. The proof of this inequality is simple under Dirichlet boundary conditions. The situation is much more complicated in the case of traction conditions on part of the boundary. We refer to Kondratiev and Oleinik in~\cite{1989KondratievV_OleinikO-aa,1989KondratievV_OleinikO-ab}, Duvaut and Lions~\cite{1976DuvautLions-a}, Nitsche~\cite{1981Nitsche-a}, and Bramble~\cite{2003Bramble-a} for proofs of various types of Korn's inequalities.
An important consequence of~\eqref{e:korn-0} is the following lemma which shows the coercivity of $a(\cdot,\cdot)$. For completeness, we include a proof of this well known result following ~\cite{1997Arnold-a} (see Appendix~\ref{apx:lemma}).
\begin{lemma}\label{l:korn-11}[\cite{1997Arnold-a}, p.~27]
	Let  $\mathfrak{R}$ be the space of rigid body motions,
	\[
	\mathfrak{R}=\left\{\vec{c}+\mathfrak{m}\vec{x}\;|\;\vec{c}\in\mathbb{R}^d,\quad \mathfrak{m}\in \mathfrak{so}(d)\right\},
	\]
 where $\vec{x}$ is the position vector in $\mathbb{R}^d$
	and $\mathfrak{so}(d)$ is the algebra of the real and anti-symmetric  
	$d\times d$ matrices.  Then it holds that
	\begin{equation}\label{e:korn-11}
		\|\nabla \vec{u}\| \lesssim  \|\varepsilon(\vec{u})\|,\quad\forall\vec{u}\in [H^1_D(\Omega)]^d\cup\big([H^1(\Omega]^d\cap\mathfrak{R}^{\perp_{L^2}}\big).
	\end{equation}
\end{lemma}
 
We exploit the fact that the linear elasticity problem for large $\lambda$ can be viewed as a penalty formulation of a constrained minimization problem (see, e.g. \cite{1986GiraultRaviart-a,2013BoffiBrezziFortin-a,2003Bramble-a,1967Necas-a}).
We introduce the subspace of divergence-free functions,
\[
W := \operatorname{Ker}(B)=\left\{\vec{v}\in V \;\big|\; B\vec{v}=\operatorname{div} \vec{v} = 0\right\}.
\]
As $B=\operatorname{div}$ is a continuous operator, its kernel is a closed subspace of $V$.  Then, \eqref{e:korn-11} and the equivalence between~\eqref{e:inf-sup} and \eqref{e:necas} imply that $a(\cdot,\cdot)=(\varepsilon(\cdot),\varepsilon(\cdot))$ is an inner product on $V$ with corresponding norm equivalent to the $\|\cdot\|_{[H^1(\Omega)]^d}$ norm. This yields the following inf-sup conditions, equivalent to~\eqref{e:inf-sup}:
\begin{equation}\label{e:inf-sup-1}
	\begin{aligned}
		\inf_{q\in Q}\sup_{\vec{v}\in V}\frac{(\operatorname{div} \vec{v},q)}{\|\varepsilon(\vec{v})\|\|q\|}=
		\inf_{\vec{v}\in W^{\perp}}\sup_{q\in Q}\frac{(\operatorname{div} \vec{v},q)}{\|\varepsilon(\vec{v})\|\|q\|}\ge \beta>0.
	\end{aligned}
\end{equation}
Here, the orthogonality in $W^{\perp}$ is in terms of the inner product $a(\cdot,\cdot)$. For the proof of the equivalence between the two conditions in~\eqref{e:inf-sup-1}, we refer to
Girault and Raviart~\cite[Lemma~4.1]{1986GiraultRaviart-a}.

Next, we define $P:V\to W$ to be the orthogonal projection onto $W$ with respect to $a(\cdot,\cdot)$. In other words,  for $\vec{v}\in V$ the projection $P\vec{v}\in W$ is the unique solution to
\begin{equation}\label{e:projection}
	a(P\vec{v},\vec{w}) = a(\vec{v},\vec{w}), \quad \forall\vec{w}\in W.
\end{equation}
%Further, for a given $\vec{f}\in V'$ let us define $\vec{u}_0$ as the unique solution to
%\begin{equation}\label{e:u0-def}
%	a(\vec{u}_0,\vec{v})=(\varepsilon(\vec{u}_0),\varepsilon(\vec{v}))=\langle \vec{f},\vec{v}\rangle\quad\mbox{for all}\quad \vec{v}\in V.
%\end{equation}
It is immediate to see that $\vec{v}_0=P\vec{v}$ solves the Stokes' equation:\\
Find $(\vec{v}_0,p)\in V\times Q$ such that
\begin{equation}\label{stokes-b-forms}
	\begin{aligned}
		   a( \vec{v}_0 , \vec{w} ) + b( \vec{w} ,p) &= a( \vec{v}, \vec{w} ),\quad &&\forall\vec{w}\in V,\\    
     b( \vec{v}_0 ,q)&=0,\quad &&\forall q\in Q,
	\end{aligned}
\end{equation}
where  the ``pressure" 
 $p$ serves as a Lagrange multiplier for the divergence free constraint.
%Indeed, we have $\vec{u}\in W$ by definition. 
%Taking $\vec{v}\in W$ in the first equation, yields
%\[
%a(\vec{u},\vec{w})=\langle \vec{f},\vec{w}\rangle=a(\vec{u}_0,\vec{w})=a(P\vec{u}_0,\vec{w}) \quad \mbox{for all}\quad \vec{w}\in W. 
%\]
%Since the projection $P\vec{u}_0$ is uniquely defined on $W$, it
%follows that $\vec{u}=P\vec{u}_0$. 
% We shall use this fact in the next section to build a preconditioner.

\section{Spectral equivalence and a robust preconditioner}\label{sec:spectral}
In this section, we use the aforementioned relationship between the inf-sup condition and Korn's inequality to develop a robust preconditioner, $M_{\lambda}$, for the linear elasticity equations. We start by proving the spectral equivalence~\eqref{eq:m-lambda-def-0} between $M_{\lambda}$ and the inverse of the linear elasticity operator, $A_\lambda$. 
\begin{theorem}\label{t:spectral-equivalence}
	If $M_{\lambda}:V'\to V$ is defined by
	\begin{equation}\label{e:preconditioner}
		M_{\lambda}=\frac{\lambda}{1+\lambda}PA^{-1} +\frac{1}{\lambda+1} A^{-1},
	\end{equation}
	where $A_\lambda$, $A$, and $P$ are as in~\eqref{e:operators}, \eqref{a-0} and \eqref{e:projection}. Then,
	\begin{equation}\label{e:spectral-equivalence}
		\langle \vec{g},M_\lambda \vec{g}\rangle \eqsim 
		\langle \vec{g},A_{\lambda}^{-1} \vec{g}\rangle \quad\forall \vec{g}\in V'.
	\end{equation}
\end{theorem}
\begin{proof}
%Clearly $M_\lambda$ is invertible, because $A$ is invertible and $P$ is positive semidefinite. 
To show the spectral equivalence, it suffices to prove that for any $\lambda\in[0,\infty)$ and $\vec{v}\in V$,
	\begin{equation}\label{e:invert}
		\langle A_\lambda \vec{v},\vec{v} \rangle \lesssim
		\langle M_\lambda^{-1} \vec{v},\vec{v} \rangle \lesssim \langle A_\lambda \vec{v},\vec{v} \rangle .
	\end{equation}
	It follows from $P^2=P$ that
	\begin{equation}\label{e:inverse-idempotent}
		(I+t P)^{-1} =  I-\frac{t}{t+1} P,\quad\forall  t\in\mathbb{R}\backslash\{-1\}.
	\end{equation}
 Using this fact, we obtain
	\begin{equation}\label{ee}
		\begin{aligned}
			M_{\lambda}^{-1} & =
		\left(\frac{1}{\lambda+1}\left(I+\lambda P\right)A^{-1}\right)^{-1}\\
  &=A\left((\lambda+1)I-\lambda P\right)\\
			& = A +\lambda A(I-P).
		\end{aligned}
	\end{equation}
	By comparing \eqref{ee} with $\langle A_{\lambda}\vec{v},\vec{v}\rangle=
	\langle A \vec{v},\vec{v}\rangle + \lambda \|\operatorname{div} \vec{v}\|^2$,
 it remains to show that
	\begin{equation}\label{eeee}
 \begin{aligned}
		&\|\operatorname{div} \vec{v}\|^2\eqsim
		\langle A (I-P) \vec{v},\vec{v}\rangle\\
  &\quad=a\left(\vec{v}-P\vec{v},\vec{v}-P\vec{v}\right) = \left\|\varepsilon\left(\vec{v}-P\vec{v}\right)\right\|^2.
  \end{aligned}
	\end{equation}
	The lower bound in~\eqref{eeee} directly follows from
	\begin{align*}
	    \|\operatorname{div} \vec{v}\|&=\sup_{q\in Q,~\|q\|=1}({\rm div}\vec{v},q)=
	\sup_{q\in Q,~\|q\|=1}({\rm div}(\vec{v}-P\vec{v}),q)\\
 &\leq 
	\left\|{\rm tr} [\varepsilon(\vec{v}-P\vec{v})]\right\|
	\leq\sqrt{d} \left\|\varepsilon(\vec{v}-P\vec{v})\right\|.
	\end{align*}
	The upper bound is just a restatement of~\eqref{e:inf-sup-1} because $(I-P)\vec{v}\in W^{\perp}$. In fact, we have
	\begin{align*}
		\beta \left\|\varepsilon\left(\vec{v}-P\vec{v}\right)\right\|&\leq 
		\sup_{q\in Q,~\|q\|=1}\big({\rm div} (\vec{v}-P\vec{v}),q\big)\\ &=\sup_{q\in Q,~\|q\|=1}\big({\rm div} \vec{v},q\big)=\left\|\operatorname{div} \vec{v}\right\|,
	\end{align*}
	and this completes the proof of \eqref{eeee} and, hence,  \eqref{e:spectral-equivalence}.    
\end{proof}

%\begin{remark}
%Theorem~\ref{t:spectral-equivalence} shows that a stable way to discretize elasticity in primal formulation is to take a stable Stokes element and use the velocity space as a underlying finite-element space. We also note that this theorem might be useful in designing auxiliary space preconditioners for the elasticity equation discretized by any space.
%\end{remark}

\subsection{Periodic boundary conditions}\label{sec:periodic}
To expose the main idea for the preconditioner and motivate how to tackle a more general case,
we investigate the case of periodic boundary conditions. Here, the spectral equivalence is in fact an equality (``$\eqsim$'' in~\eqref{e:spectral-equivalence} becomes ``$=$'').
For a given $\vec{f}$ (periodic in all $d$ directions),
we extend the solution $\vec{u}$ to $\mathbb{R}^d$ by periodicity and use a Fourier transform. We define
$\vec{u}_\lambda$, $\vec{u}_{\infty}$, and $\vec{u}_0$ as the solutions to
\[
A_\lambda \vec{u}_\lambda = \vec{f}, \quad 
A_\infty \vec{u}_\infty = \vec{f}, \quad \text{ and }
A \vec{u}_0 = \vec{f},
\]
respectively.  Here,  $A_{\infty}$ is the operator corresponding to the Stokes equation (with periodic boundary conditions) and $\vec{u}_\infty$ is the velocity component of its solution.  To show that $M_\lambda=A_{\lambda}^{-1}$, we prove that 
\begin{equation}\label{spectral-equivalence-u-p}
	\vec{u}_{\lambda}=
	\frac{\lambda}{\lambda+1} \vec{u}_{\infty} + 
	\frac{1}{\lambda+1} \vec{u}_{0}.
\end{equation}
The proof of this relation is a straightforward computation using the Fourier transform and the following identities:
\begin{eqnarray}
	&&   \widehat{\operatorname{div}\varepsilon( \vec{w} )} = \frac12|\xi|^2(I+\Pi_{\xi})\widehat{ \vec{w} }, \label{symbols-0}\\
	&&   \widehat{\operatorname{div}\mathcal{C}\varepsilon( \vec{w} )} = \frac12|\xi|^2\left(I+(2\lambda+1)\Pi_{\xi}\right)\widehat{ \vec{w} }, \label{symbols-1}
\end{eqnarray}
where $\Pi_{\xi} = |\xi|^{-2} \xi\xi^*$, $\mathcal{C}(X) =\frac{1}{2}(X+X^*) + \lambda \operatorname{tr}(X) I$, and $X\in \mathbb{R}^{d\times d}$.
Notice that $\Pi_{\xi}^2 = \Pi_{\xi}$ and, hence,~\eqref{e:inverse-idempotent} holds with $\Pi_\xi$ instead of $P$.
We then find that 
\begin{equation}\label{u-lambda-0}
	\widehat{ \vec{u} }_{\lambda} = 2|\xi|^{-2} \left(I-\frac{2\lambda+1}{2(\lambda+1)}\Pi_\xi\right) \widehat{ \vec{f} }, \qquad
	\widehat{ \vec{u} }_{0} = 2|\xi|^{-2} \left(I-\frac{1}{2}\Pi_\xi\right) \widehat{ \vec{f} }.
\end{equation}
Furthermore, the Stokes' problem in the Fourier space is:
\begin{equation}
	\begin{pmatrix}
		\frac12|\xi|^{2}(I+\Pi_\xi) & \xi \\
		\xi^* & 0
	\end{pmatrix}
	\begin{pmatrix}
		\widehat{ \vec{u} }_\infty\\ \widehat{p}
	\end{pmatrix}=
	\begin{pmatrix}
		\widehat{ \vec{f} }\\ 0
	\end{pmatrix}.
\end{equation}
Solving this system shows that
\begin{equation}\label{fourier-solve}
	\widehat{ \vec{u} }_{\infty} = 2|\xi|^{-2} (I-\Pi_\xi) \widehat{ \vec{f} },
	\qquad \widehat{p} = 
	|\xi|^{-2}\xi^* \widehat{ \vec{f} }.
\end{equation}
Finally, the relation~\eqref{spectral-equivalence-u-p} follows immediately from~\eqref{u-lambda-0} and
\eqref{fourier-solve}.

\subsection{Discrete problems}
Although we have defined the preconditioner $M_\lambda$ for $A_\lambda$ on the continuous level, a quick check shows that the analysis in Theorem \ref{t:spectral-equivalence} holds verbatim for the discretized problem as long as a Stokes stable finite-element pair, $V_h\times Q_h\subset V\times Q$, is available. In particular, assume $V_h\times Q_h$ satisfies the discrete inf-sup condition (cf.~\cite{1986GiraultRaviart-a}),
\begin{equation}\label{e:inf-sup-discrete}
	\inf_{q_h\in Q_h}\sup_{\vec{v}_h\in V_h}\frac{(\operatorname{div} \vec{v}_h,q_h)}{\|\nabla \vec{v}_h\|\|q_h\|}\ge\beta_h>0,
\end{equation}
and let $A_\lambda^h: V_h\to V_h^\prime$ and $\vec{u}_h\in V_h$ be given by 
\begin{equation}\label{eq:primal:discrete}
\langle A_\lambda^h\vec{u}_h,\vec{v}_h\rangle=a_\lambda(\vec{u}_h,\vec{v}_h)=\langle\vec{f},\vec{v}_h\rangle\quad\forall\vec{v}\in V_h.
\end{equation}

%{\color{red}
\begin{remark}
Under the assumption ${\rm div}V_h\subset Q_h$, \eqref{eq:primal:discrete} has uniform a priori error estimates for all $\lambda\in[0,\infty)$. Here we refer to \cite{2014GuzmanNeilan,2017JohnLinkeMerdonNeilanRebholtz,2018ChristiansenSHuK-a,2020FuGuzmanNeilan} for stable Stokes' element pairs satisfying ${\rm div}V_h\subset Q_h$ and using discontinuous pressure spaces. We need to be careful when ${\rm div}V_h\not\subset Q_h$ and  modify the bilinear form $a_h^\lambda$ in such cases as follows:
\begin{equation}
    \langle A_\lambda^h\vec{u}_h,\vec{v}_h\rangle=a(\vec{u}_h,\vec{v}_h)+\lambda b(\vec{v}_h,\Pi_h{\rm div}\vec{u}_h)=\langle\vec{f},\vec{v}_h\rangle\quad\forall\vec{v}\in V_h,
    \end{equation}
where $\Pi_h$ is the $L^2$-projection onto $Q_h$.  This approach has been discussed in \cite{2007Braess,1999Schoberl}, where the role of $\Pi_h$ is implicit but crucial to drawing the connection with the Stokes' equations, thus ensuring robust a priori error estimates. In general, the action of $\Pi_h$ is computed by inverting a mass matrix, which could be costly, especially when the functions in $Q_h$ are subject to inter-element continuity constraints. It is, however, easy to justify that we can use a spectrally equivalent diagonal matrix, such as the diagonal of the mass matrix, to implement the action of $\Pi_h$. This is the approach we have taken in the numerical tests for the Taylor-Hood~\cite{1973TaylorCHoodP-a} ($\mathcal{P}_2\times\mathcal{P}_1$) element as presented in Section~\ref{sec:results}.
\end{remark}

Finally, let $W_h=\{\vec{v}_h\in V_h: b(\vec{v}_h,q_h)=0~\forall q_h\in Q_h\}$, $A_h=A_0^h$, and $P_h: V_h\to W_h$ be the $a(\cdot,\cdot)$ orthogonal projection. A proof, analogous to the proof of Theorem \ref{t:spectral-equivalence}, then leads to 
\begin{equation}\label{e:discpreconditioner}
		\left (A^h_{\lambda}\right )^{-1}\eqsim M^h_{\lambda}:=\frac{\lambda}{1+\lambda}P_hA_h^{-1} +\frac{1}{\lambda+1} A_h^{-1}.
	\end{equation}

An immediate, and important, observation is that computing the action of the preconditioner, $M_{\lambda}^h$, does not require a evaluating
$P_h\vec{v}_h$ for some $\vec{v}_h\in V_h$ directly. This would be difficult, as a basis in the weakly divergence-free space, $W_h$, is not always available. Instead, the action of $P_h$ is computed by solving the discrete Stokes' problem:\\ Find $P_h\vec{v}_h\in V_h$, $p_h\in Q_h$ such that
\begin{equation}\label{stokes-discrete}
	\begin{aligned}
		   a( P_h\vec{v}_h , \vec{w}_h ) + b( \vec{w}_h ,p_h) &= a( \vec{v}_h, \vec{w}_h ),\quad &&\forall\vec{w}_h\in V_h,\\    
     b(P_h\vec{v}_h ,q_h)&=0,\quad &&\forall q_h\in Q_h.
	\end{aligned}
\end{equation}
We add that the analysis we have given here provides a theoretical justification for the preconditioning results reported in~\cite{2013KarerKrausZikatanov-a}.

\section{Numerical Results}\label{sec:results}
%\textcolor{red}{Maybe the summary goes here instead?}

In this section, we provide numerical examples demonstrating the effectiveness and robustness of the preconditioner, $M^h_{\lambda}$, as defined in~\eqref{e:discpreconditioner}. The computational domain is $\Omega = [0,1] \times [0,1]$, and we seek to solve~\eqref{elasticity-b-forms} for the exact solution $\vec{u}$ given by
\begin{equation*}
	\vec{u} =
	\begin{pmatrix}
		\sin(\pi x) \cos(\pi y), \
		-\cos( \pi x) \sin(\pi y)
	\end{pmatrix}.
\end{equation*}
We compute the right-hand side, $\vec{f}$, accordingly and impose pure Dirichlet boundary conditions, i.e., $\Gamma = \Gamma_D$. Equations are discretized on a uniformly refined triangular mesh with mesh size $h=2^{-L}$, where we use continuous and piecewise quadratic $\mathbb{R}^d$-valued polynomials in $\mathcal{P}_2$ to approximate $\vec{u}$. We test different finite-element spaces for the multiplier $p$.

The resulting linear system of equations is solved by the preconditioned conjugate gradient method, with $M_{\lambda}^h$ as the preconditioner. We implement the actions of $P_hA_h^{-1}$ and $A_h^{-1}$ using direct solvers. The stopping criterion is based on the relative residual with tolerance $10^{-6}$.  All numerical experiments, including the
discretization and the preconditioned linear solvers, were implemented using the finite-element and solver library \verb|HAZmath|~\cite{hazmath}.

For the first set of experiments, we employ the space of piecewise constants, $\mathcal{P}_0$, as the finite-element space for $p$. Thus, we implement the action of $P_hA_h^{-1}$ by solving Stokes' equations~\eqref{stokes-b-forms} using the $\mathcal{P}_2\times\mathcal{P}_0$ finite-element pair, which is known to be inf-sup stable in 2D. We report the performance of the proposed preconditioner in Table~\ref{tab:2D-P2P0-iter} (for the number of iterations) and Table~\ref{tab:2D-P2P0-cond} (for the condition number of $M^h_\lambda A^h_\lambda$). These results show that the number of iterations and condition number remain stable as $\nu \to 0.5^-$, i.e., as $\lambda \to \infty$. This observation confirms our theoretical predictions.

\begin{table}[h!]
\begin{center}
\caption{Number of iterations for $\mathcal{P}_2\times\mathcal{P}_0$} \label{tab:2D-P2P0-iter}
\begin{tabular}{c | c c c c c }
\hline \hline 
$h=2^{-L}$	 & $\nu = 0.25$ & $\nu = 0.4$ & $\nu = 0.49$ & $\nu = 0.499$  &$\nu = 0.4999$\\ \hline 
$L=2$  &  4  &  5  &  6  &  6  &  6  \\
$L=3$  &  3  &  4  &  6  &  7  &  7  \\
$L=4$ &  3  &  4  &  6  &  7  &  7  \\
$L=5$ &  3  &  4  &  6  &  7  &  7  \\
$L=6$ &  3  &  4  &  5  &  7  &  7  \\
\hline \hline 
\end{tabular}
\end{center}
\end{table}

\begin{table}[h!]
	\begin{center}
	\caption{Condition number of $M^h_\lambda A^h_\lambda$ for $\mathcal{P}_2\times\mathcal{P}_0$} \label{tab:2D-P2P0-cond}
		\begin{tabular}{c | c c c c c }
			\hline \hline 
$h=2^{-L}$& $\nu = 0.25$ & $\nu = 0.4$ & $\nu = 0.49$ & $\nu = 0.499$  & $\nu = 0.4999$\\ \hline 
			$L=2$  &  1.15  &  1.48  &  2.52  &  2.84  &  2.88  \\
			$L=3$  &  1.14  &  1.44  &  2.47  &  2.98  &  3.03  \\
			$L=4$ &  1.13  &  1.44  &  2.55  &  2.90  &  2.94  \\
			$L=5$ &  1.13  &  1.44  &  2.51  &  2.86  &  2.89  \\
			$L=6$ &  1.13  &  1.44  &  2.45  &  2.87  &  2.91  \\
			\hline \hline 
		\end{tabular}
	\end{center}
\end{table}

For the second set of experiments, we utilize the nodal element space, $\mathcal{P}_1$, for $p$. In this case, as we pointed out earlier, assembling $b(\vec{v}_h,\Pi_h{\rm div} \vec{u}_h)$  requires inverting $\mathcal{M}_h$, the $\mathcal{P}_1$ mass matrix. In the tests, however, we use the inverse of ${\rm diag}(\mathcal{M}_h)$ to approximate $\mathcal{M}_h^{-1}$. Then the action of $P_hA_h^{-1}$ is computed  by solving \eqref{stokes-discrete} using the $\mathcal{P}_2\times\mathcal{P}_1$ finite-element pair (Taylor-Hood). The performance of $M_{\lambda}^h$ is presented in Table~\ref{tab:2D-P2P1-iter} (for the number of iterations) and Table~\ref{tab:2D-P2P1-cond} (for the condition number of $M^h_\lambda A^h_\lambda$). Although the number of iterations and condition number is slightly higher than those obtained with the $\mathcal{P}_2\times\mathcal{P}_0$ finite-element pair, they remain stable as $\nu \to 0.5^-$, i.e., as $\lambda \to \infty$. This indicates that the efficacy of the proposed preconditioner $M_{\lambda}^h$ is not affected by the choice of finite-element space for $p$, as long as the corresponding finite-element pair is inf-sup stable.

\begin{table}[h!] 
	\begin{center}
\caption{Number of iterations for $\mathcal{P}_2\times\mathcal{P}_1$} \label{tab:2D-P2P1-iter}
		\begin{tabular}{c | c c c c c }
			\hline \hline 
$h=2^{-L}$	& $\nu = 0.25$ & $\nu = 0.4$ & $\nu = 0.49$ & $\nu = 0.499$  & $\nu = 0.4999$\\ \hline 
			$L=2$  &  4  &  5  &  5  &  5  &  5  \\
			$L=3$  &  4  &  6  &  11  &  12  &  12  \\
			$L=4$ &  4  &  6  &  12  &  15  &  15  \\
			$L=5$ &  4  &  6  &  12  &  15  &  15  \\
			$L=6$ &  4  &  6  &  11  &  14  &  15  \\
			\hline \hline 
		\end{tabular}
	\end{center}
\end{table}

\begin{table}[h!] 
	\begin{center}
	\caption{Condition number $M^h_\lambda A^h_\lambda$ for $\mathcal{P}_2\times\mathcal{P}_1$} \label{tab:2D-P2P1-cond}
		\begin{tabular}{c | c c c c c }
			\hline \hline 
$h=2^{-L}$	& $\nu = 0.25$ & $\nu = 0.4$ & $\nu = 0.49$ & $\nu = 0.499$  & $\nu = 0.4999$\\ \hline 
			$L=2$  &  1.20  &  1.71  &  4.31  &  5.69  &  5.89  \\
			$L=3$  &  1.20  &  1.71  &  4.38  &  5.81  & 6.02  \\
			$L=4$ &  1.19  &  1.71  &  4.38  &  5.81  & 6.02  \\
			$L=5$ &  1.18  &  1.71  &  4.38  &  5.81  & 6.02  \\
			$L=6$ &  1.17  &  1.71  &  4.38  &  5.81  & 6.02  \\
			\hline \hline 
		\end{tabular}
	\end{center}
\end{table}

%\begin{table} 
%	\begin{center}
%		\caption{3D: $P2$-$P1$}
%		\begin{tabular}{c | c c c c c }
%			\hline \hline 
%$h=2^{-L}$	& $\nu = 0.25$ & $\nu = 0.4$ & $\nu = 0.49$ & $\nu = 0.499$  & $\nu = 0.4999$\\ \hline 
%			$L=2$  &  4  &  5  &  5  &  5  &  5  \\
%			$L=3$  &  4  &  6  &  11  &  12  &  12  \\
%			$L=4$ &  4  &  6  &  12  &  15  &  15  \\
%			$L=5$ &  4  &  6  &  12  &  15  &  15  \\
%			$L=6$ &  4  &  6  &  11  &  14  &  15  \\
%			\hline \hline 
%		\end{tabular}
%	\end{center}
%\end{table}
%
%\begin{table} 
%	\begin{center}
%		\caption{3D: $P2$-$P1$}
%		\begin{tabular}{c | c c c c c }
%			\hline \hline 
%$h=2^{-L}$   $h=2^{-L}$ & $\nu = 0.25$ & $\nu = 0.4$ & $\nu = 0.49$ & $\nu = 0.499$  & $\nu = 0.4999$\\ \hline 
%			$L=2$  &  1.21  &  1.71  &  4.31  &  5.69  &  5.89  \\
%			$L=3$  &  1.19  &  1.71  &  4.38  &  5.81  & 6.02  \\
%			$L=4$ &  1.19  &  1.71  &  4.38  &  5.81  & 6.02  \\
%			$L=5$ &  1.18  &  1.71  &  4.38  &  5.81  & 6.02  \\
%			$L=6$ &  1.17  &  1.71  &  4.38  &  5.81  & 6.02  \\
%			\hline \hline 
%		\end{tabular}
%	\end{center}
%\end{table}

\section{Concluding Remarks}\label{sec:conclusion}

Theorem~\ref{t:spectral-equivalence} and the corresponding numerical tests in Section \ref{sec:results} confirm that 
% a stable way to discretize elasticity in primal formulation is to take a stable Stokes element and use the velocity space as a underlying finite-element space. From the presented analysis it is clear that 
a robust discretization and solvers for nearly incompressible elasticity must rely on robust solvers for Stokes' equations. This point is important as it confirms the relationship between the inf-sup condition for Stokes' equation and the second Korn's inequality for linear elasticity.  Numerical results show that the proposed preconditioner, $M_{\lambda}^h$, remains stable as $\nu \to 0.5^-$, i.e., as $\lambda \to \infty$, regardless of the choice of finite-element space for $p$, as long as it forms a Stokes' inf-sup stable finite-element pair. 
While the preconditioner is robust for several families of finite elements, in our view, the best suited ones are elements recently developed in~\cite{2019ChristiansenSHuK-a,2018ChristiansenSHuK-a}, as they provide spaces with projections that commute with the divergence and lead directly to discretizations for linear elasticity. 
% However, for simplicity, we only report numerical results for well-known elements that fit in our the framework \cite{1973TaylorCHoodP-a,1985BernardiRaugel-a}.
% and we have also presented numerical tests on discretizations resulting from Taylor-Hood~}.
Finally, we note that this work would be useful in designing auxiliary space preconditioners for the elasticity equation when discretized using (any) stable finite-element space. 

%% do not use (ltz1): \input{algebraic.tex}

%\appendix
%
%\section{An auxiliary result about the Korn's inequality}
%In this section we show how inequality~\eqref{e:korn-0} implies the \eqref{e:korn-1} for the function spaces of interest. We use the proof as given by D.~N.~Arnold in~\cite{1997Arnold-a}.
%

\section*{Acknowledgements}
The work of Adler and Hu is partially supported by the National Science Foundation (NSF) under grant DMS-2208267. 
The research of Zikatanov is supported in part by the U.~S.-Norway Fulbright Foundation and the U.~S. National Science Foundation grant DMS-2208249. 

\renewcommand{\appendix}{\par
 \setcounter{section}{0}
 \setcounter{subsection}{0}
  \gdef\thesection{\Alph{section}}
}
\appendix

\section{Proof of Lemma~\ref{l:korn-11}}\label{apx:lemma} 

\begin{applemma}\label{l:korn-12}
	Let  $\mathfrak{R}$ be the space of rigid body motions
	\[
	\mathfrak{R}=\left\{\vec{c}+\mathfrak{m}\vec{x}\;|\;\vec{c}\in\mathbb{R}^d,\quad \mathfrak{m}\in \mathfrak{so}(d)\right\},
	\]
 where $\vec{x}$ is the position vector in $\mathbb{R}^d$
	and $\mathfrak{so}(d)$ is the algebra of the real and anti-symmetric  
	$d\times d$ matrices.  Then it holds that
	\begin{equation}\label{e:a-korn-11}
		\|\nabla \vec{u}\| \lesssim  \|\varepsilon(\vec{u})\|,\quad\forall\vec{u}\in [H^1_D(\Omega)]^d\cup\big([H^1(\Omega]^d\cap\mathfrak{R}^{\perp_{L^2}}\big).
	\end{equation}
\end{applemma}
\begin{proof}  
	First, we only consider $\vec{u}\in V=[H^1_D(\Omega)]^d$ and remark that the proof for the case when $\vec{u}\in [H^1(\Omega]^d\cap\mathfrak{R}^{\perp_{L^2}}$ is similar and simpler.
	
	To start, assume that~\eqref{e:a-korn-11} is not true. Then, there exists a sequence $\{\vec{v}_n\}\subset V$ such that $\|\nabla \vec{v}_n\|=1$ and $\|\varepsilon(\vec{v}_n)\|\le\frac{1}{n}$. From the Poincar\'e inequality, we conclude that $\{\vec{v}_n\}$ is a bounded sequence in $L^2(\Omega)$. Next, since the embedding $V=[H^1_D(\Omega)]^d\hookrightarrow L^2(\Omega)$ is compact, we conclude that this bounded sequence has a subsequence convergent in $L^2(\Omega)$. We denote the subsequence again by $\{\vec{v}_n\}$.  Applying~\eqref{e:korn-0} to $\vec{u}=(\vec{v}_n-\vec{v}_m)$ for sufficiently large $n$ and $m$, we find that $\{\vec{v}_n\}$ is a Cauchy sequence in $V$ and hence, converges to some element $\vec{v}\in V$. This gives $\|\nabla \vec{v}\|=1$ and $\varepsilon(\vec{v}_n)\to 0$. Hence, $\varepsilon(\vec{v})=0$. This implies that $\vec{v}$ is a rigid body motion, namely, $\vec{v}=\mathfrak{m}\vec{x}+\vec{c}\in \mathfrak{R}$.
	
	What remains is to show that if $\Gamma_D$ has a nonzero $(d-1)$ dimensional measure, then $\vec{v}=0$. This will lead to a contradiction with the assumption that~\eqref{e:a-korn-11} does not hold. Let us pick $\vec{x}\in \Gamma_D$ such that $\Gamma_D$ is smooth in a neighborhood of $\vec{x}$. For the case of a polyhedral domain, which we consider here, take $\vec{x}$ in the interior of a planar face of $\Gamma_D$.  For any $\vec{y}$ that is in this planar face, we have $\mathfrak{m}(\vec{x}-\vec{y})=0$.  Since the face is of dimension $(d-1)$, it follows that $\mathfrak{m}$ has at least a $(d-1)$-dimensional kernel. However, $\mathfrak{m}$ is antisymmetric and real, and all its nonzero eigenvalues are pure imaginary and are complex conjugate to each other, that is, the nonzero eigenvalues come in pairs. Hence, we cannot have any nonzero eigenvalue of $\mathfrak{m}$. Thus, $\mathfrak{m}=0$ and $\vec{v}$ is a constant vector vanishing on $\Gamma_D$. We then conclude that $\vec{v}=0$ which contradicts $\|\nabla \vec{v}\|=1$ and shows~\eqref{e:a-korn-11}.
\end{proof}

\bibliographystyle{elsarticle-num}
\bibliography{bib_elast0}

\end{document}